\documentclass[12pt]{amsart}
\usepackage{enumitem, verbatim}
\usepackage{color}
\usepackage{appendix}
\usepackage{amsmath, amssymb, amsthm}
\usepackage{mathrsfs}

\usepackage[margin=2 cm,heightrounded=true,centering]{geometry}
\usepackage{graphicx}
\usepackage{wrapfig}
\usepackage{cancel}

\newtheorem{theorem}{Theorem}[section]
\newtheorem{proposition}[theorem]{Proposition}
\newtheorem{lemma}[theorem]{Lemma}

\newtheorem{definition}[theorem]{Definition}

\DeclareMathOperator{\ric}{Ric}

\DeclareMathOperator{\hess}{Hess}
\DeclareMathOperator{\vol}{Vol}

\DeclareMathOperator{\tr}{tr}
\DeclareMathOperator{\tf}{tf}

\DeclareMathOperator{\SO}{SO}

\DeclareMathOperator{\divergence}{div}

\begin{document}

\title{Static near horizon geometries and rigidity of quasi-Einstein manifolds}
\author{Eric Bahuaud}
\address[Eric Bahuaud]{Department of Mathematics, Seattle University, Seattle, WA 98122, USA}
\email{bahuaude@seattleu.edu}

\author{Sharmila Gunasekaran}
\address[Sharmila Gunasekaran]{The Fields Institute for Research in Mathematical Sciences, 222 College St, Toronto ON, Canada M5T 3J1}
\email{gunasek1@ualberta.ca}

\author{Hari K Kunduri}
\address[Hari K Kunduri]{Department of Mathematics and Statistics and Department of Physics and Astronomy, McMaster University, Hamilton ON, Canada L8S 4K1}
\email{kundurih@mcmaster.ca}

\author{Eric Woolgar}
\address[Eric Woolgar]{Department of Mathematical and Statistical Sciences and Theoretical Physics Institute, University of Alberta, Edmonton, Alberta, Canada T6G 2G1}
\email{ewoolgar@ualberta.ca}

\begin{abstract}
Static vacuum near horizon geometries are solutions $(M,g,X)$ of a certain quasi-Einstein equation on a closed manifold $M$, where $g$ is a Riemannian metric and $X$ is a closed 1-form. It is known that when the cosmological constant vanishes, there is rigidity: $X$ vanishes and consequently $g$ is Ricci flat. We study this form of rigidity for all signs of the cosmological constant. It has been asserted that this rigidity also holds when the cosmological constant is negative, but we exhibit a counter-example. We show that for negative cosmological constant if $X$ does not vanish identically, it must be incompressible, have constant norm, and be nontrivial in cohomology, and $(M,g)$ must have constant scalar curvature and zero Euler characteristic. If the cosmological constant is positive, $X$ must be exact (and vanishing if $\dim M=2$). Our results apply more generally to a broad class of quasi-Einstein equations on closed manifolds. We extend some known results for quasi-Einstein metrics with exact 1-form $X$ to the closed $X$ case. We consider near horizon geometries for which the vacuum condition is relaxed somewhat to allow for the presence of a limited class of matter fields. An appendix contains a generalization of a result of Lucietti on the Yamabe type of quasi-Einstein compact metrics (with arbitrary $X$).
\end{abstract}

\maketitle

\section{Introduction}
\setcounter{equation}{0}

\noindent A spacetime is said to possess a black hole if it has two disjoint open regions, one of which can communicate with infinity by causal signals and one of which cannot, connected by a boundary. The boundary is the event horizon. For static and, more generally, stationary black holes, the event horizon is also a Killing horizon, meaning that there is a linear combination of Killing vector fields which is null at the horizon and timelike outside it. If the Killing horizon has vanishing surface gravity (zero temperature), it is called degenerate and the black hole is called extreme.

The near horizon geometry equations describe the geometry of degenerate Killing horizons in spacetime $(\mathbf{M}, \mathbf{g})$. Through a limiting procedure at the Killing horizon, one can construct a spacetime metric of the form
\begin{equation}
\label{eq1.1}
\mathbf{g} = 2dv\left ( dr +rX_i(x) dx^i +\frac12 r^2 Y(x)dv\right ) +g_{ij}(x) dx^idx^j
\end{equation} that obeys the Einstein equations. The coordinates $x^i$ parametrize a closed Riemannian submanifold $(M,g)$ of dimension $n$ (then the spacetime has dimension $n+2$). 
If the original spacetime is static, so that $\partial_v$ is timelike outside the horizon and hypersurface-orthogonal, then the 1-form $X$ is closed, so $dX=0$, and furthermore $dY=YX$. For a review, see \cite{KL}. 

Rigidity results for near horizon geometries inform uniqueness results for extreme black holes of corresponding spacetimes. 

The static vacuum near horizon geometry problem reduces to the problem of finding a metric of constant 2-Bakry-\'Emery Ricci curvature with closed 1-form $X$. We define the $m$-Bakry-\'Emery Ricci curvature as follows.
\begin{definition}\label{definition1.1}
Let $(M,g)$ be a Riemannian manifold. For $m\in (0,\infty)$ and a 1-form $X$, the \emph{$m$-Bakry-\'Emery Ricci curvature} $\ric_X^m$ is defined by
\begin{equation}
\label{eq1.2}
\ric_X^m:= \ric +\frac12 \pounds_X g -\frac{1}{m}X\otimes X.
\end{equation}
The \emph{quasi-Einstein equation} is the equation
\begin{equation}
\label{eq1.3}
\ric_X^m=\lambda g,\ \lambda\in{\mathbb R}.
\end{equation}
A \emph{quasi-Einstein manifold} is the triplet $(M,g,X)$ satisfying \eqref{eq1.3}. If $m=2$, the quasi-Einstein manifold is called a \emph{vacuum near horizon geometry} and if in addition, $dX =0$, then it is a called a \emph{static vacuum near horizon geometry.}
\end{definition}
In a slight abuse of notation, we use the same symbol for the 1-form $X$ and its metric-dual vector field (we also do not use distinct symbols for the spacetime $X$ of \eqref{eq1.1} and its pullback to $(M,g)$ in \eqref{eq1.2}). There is an extended definition of $m$-Bakry-\'Emery Ricci curvature for $m\in {\mathbb R}\cup\{ \pm\infty\}$, but the extension is not useful for our present purposes. We use the term \emph{quasi-Einstein equation} in a somewhat restricted sense, in that we require $m$ and $\lambda$ to be constants, with $m \in (0,\infty)$ and $\lambda \in \mathbb{R}$ (other authors permit these quantities to be functions, e.g., \cite{FFT}, while some authors require $m$ and $\lambda$ to be constant and $X$ to be exact, e.g., \cite{CSW}).

Many results that are true for near horizon geometries are also true for $m$-quasi-Einstein manifolds with $m$ taking values on an interval containing $m=2$. Our goal is to find solutions $(M,g,X)$ of \eqref{eq1.3} for  a closed manifold $M$ and a closed 1-form $X$ and in particular to determine the conditions under which such solutions must be Einstein. 
We extend some results of \cite{CSW} in the case of $X$ an exact (sometimes called gradient) 1-form and of \cite{CRT} when $m=2$ (see also \cite[Theorem 1.1]{Case}). In \cite{CRT}, it is shown that when $\lambda=0$ then $X$ vanishes and $(M,g)$ is Einstein. It was also claimed in \cite{CRT} that the same result holds for $\lambda<0$, though an explicit proof was not given. The claim became Theorem 4.1 of \cite{KL}.\footnote
{The reader should note that \cite{KL} appears in a journal that permits articles to be updated after publication. We refer here to Theorem 4.1 of \cite{KL} as it exists at the time that we are writing.}
However, a counter-example to this claim exists. Lim \cite{Lim} searched for solutions of the quasi-Einstein equation on homogeneous 3-manifolds. One of the solutions that she found yields a non-Einstein static vacuum near horizon geometry with $\lambda<0$. Let $M={\mathbb S}^1\times \Sigma$, where $\Sigma$ is a compact hyperbolic 2-manifold with metric $g_{\Sigma}$ of sectional curvature $-m$. Endow $M$ with the metric
\begin{equation}
\label{eq1.4}
g=d\Phi^2\oplus g_{\Sigma} =\frac{1}{m^2}(X\otimes X) \oplus g_{\Sigma},
\end{equation}
where $X=m d\Phi$ for $\Phi$ a coordinate on the ${\mathbb S}^1$ factor, so $X$ is closed but not exact since $\Phi$ is multi-valued on any open patch covering ${\mathbb S}^1$. It is easy to see that $(M,g,X)$ solves the quasi-Einstein equation with $\lambda=-m$. Setting $m=2$ yields a counter-example to the aforementioned claim. The following 5-dimensional spacetime possesses a degenerate Killing horizon whose near horizon geometry is $(M,g,X)$:
\begin{equation}
\label{eq1.5}
ds^2_5 = g_{\rm XBTZ}\oplus g_{\Sigma} = \frac{ 2dv dr}{\sqrt{r+a}} + 8r dv d\Phi +  4(r+a) d\Phi^2 + g_{\Sigma},
\end{equation}
where $a>0$ is constant. Here $g_{\rm XBTZ}$ is an extreme BTZ black hole metric (recall that any BTZ metric is locally isometric to ${\rm AdS}_3$, so the Ricci tensor of $g_{\rm XBTZ}$ is $\ric_{g_{\rm XBTZ}}=-2g_{\rm XBTZ}$; therefore, \cite[Theorem 4.1]{KL} notwithstanding, \eqref{eq1.5} is consistent with \cite[Corollary 3.1]{KL}.)
The 5-dimensional near horizon limit of \eqref{eq1.5} is obtained by making the replacements $r\mapsto \epsilon r$, $v\mapsto v/\epsilon$, and taking $\epsilon\to 0$ to yield a spacetime metric of the form \eqref{eq1.1}:
\begin{equation}
\label{eq1.6}
\mathbf{g} = \frac{ 2dv dr}{\sqrt{a}} + 8r dv d\Phi + 4ad\Phi^2  + g_{\Sigma},
\end{equation}
from which we can read off \eqref{eq1.4} by choosing $a=\frac14$.

Before stating our main result, we recall here the definitions of Yamabe constant and Yamabe invariant. For $n=2$, we define the Yamabe constant of $M$ to be $4\pi$ times the Euler characteristic. For $n\ge 3$, the Yamabe functional of a closed manifold $(M^n,g)$, evaluated on some function $\phi\in C^{\infty}(M)\backslash \{ 0 \}$, is
\begin{equation}
\label{eq1.7}
Q_g(\phi):=\frac{\int_M \left [ \frac{4(n-1)}{(n-2)}|\nabla \phi |^2 +R_g \phi^2\right ] dV_g}{\left ( \int_M \phi^{\frac{2n}{(n-2)}}dV_g\right )^\frac{n-2}{n}},
\end{equation}
where $R_g$ is the scalar curvature of $g$. The \emph{Yamabe constant} associated to a conformal class $[g]$ of $g$ is $Y(M,[g]):=\inf_{\phi\in C^{\infty}(M)\backslash \{ 0 \}}$. Finally, the \emph{Yamabe invariant} (sometimes called the sigma-constant) of $M$ is the supremum of $Y(M,[g])$ taken over conformal classes $[g]$. Clearly the only closed orientable 2-manifold of positive Yamabe type (i.e., with positive Yamabe constant) is the 2-sphere ${\mathbb S}^2$. Since a 3-manifold is of positive Yamabe type if and only if it admits a metric of positive scalar curvature, then \cite[Theorem E]{GL} or the results of \cite{SY} imply that the only closed orientable 3-manifolds of positive Yamabe type are spherical spaces, ${\mathbb S}^1\times {\mathbb S}^2$, and connected sums thereof.

We now state our main theorem, which recovers the $\lambda=0$ rigidity result of \cite{CRT}, gives an extension to nonzero $\lambda$ which accounts for the above example, and extends some of the results of \cite{CSW} to cases where $X$ is closed but not assumed to be exact. As our results hold for any $m \in (0,\infty)$, we have stated and proved them for quasi-Einstein manifolds from which the static near horizon geometry case can be obtained by applying $m=2$. Results in \cite{CSW} yield a proof of the theorem for exact $X$.

\begin{theorem}\label{theorem1.2}
Let $(M,g,X)$ be a closed manifold such that $(g,X)$ solves \eqref{eq1.3} for some $m\in (0,\infty)$ and some 1-form $X$ such that $dX=0$. 
\begin{itemize}
\item[(i)] If $\lambda>0$ then $X$ is exact and $M$ is of positive Yamabe type.
If the scalar curvature obeys $R=const$ then $X=0$ and $(M,g)$ is Einstein. 
\item[(ii)] If $\lambda=0$ then $X=0$ and $(M,g)$ is Ricci-flat.
\item[(iii)] If $\lambda<0$, then the scalar curvature is constant and either 
\begin{itemize}
\item [a)] $X$ is not exact, $n:=\dim M \ge 3$, and $(M,g,X)$ obeys
\begin{equation}
\label{eq1.8}
\divergence_g X = 0,\ |X|_g^2 = -m\lambda,\ R=\left ( n-1\right )\lambda \text{ and } \chi(M)=0
\end{equation}
 or $(M,g)$ is negative Einstein.
\item [b)] If $b_1(M)=0$ and $\lambda<0$ then $(M,g)$ is negative Einstein.
\end{itemize}
\end{itemize}
\end{theorem}

The following remarks are in order.

\begin{enumerate}

\item Nontrivial quasi-Einstein metrics with $dX=0$ do not occur if $M$ is a closed 2-manifold. Theorem \ref{theorem1.2} shows that such metrics cannot arise unless $\lambda>0$ and $X$ is exact. But in \cite[Theorem 1.2]{CSW} it is shown that there are no nontrivial quasi-Einstein metrics on closed 2-manifolds when $X$ is exact.

\item In the case of Theorem \ref{theorem1.2}.(i), the assumption that $M$ is closed is redundant, since the closure of $M$ is guaranteed (even if $X$ is not a closed form) by a theorem of Limoncu \cite{Limoncu} whenever $\lambda>0$, which also implies that the fundamental group is finite. This in turn implies that some manifolds of positive Yamabe type do not admit quasi-Einstein metrics with $\lambda>0$; e.g., ${\mathbb S^2}\times {\mathbb S^1}$ \cite{KW}.

\item In every case where we obtain rigidity, we find that $X=0$. Consequently, our conclusions provide stronger rigidity results than the conditions required for rigidity in the moduli space of metrics on a compact manifold (which would follow whenever $\pounds_X g-\frac{1}{m}X\otimes X$ is proportional to $g$).

\item It follows from Theorem \ref{theorem1.2}.(i) and (ii) that all static locally homogeneous quasi-Einstein closed manifolds with $\lambda\ge 0$ are Einstein. But the example of \cite{Lim} described above shows that this result does not extend to $\lambda<0$. 

\item A family of complete, nonhomogeneous metrics on $S^2$-bundles over K\"ahler-Einstein spaces constructed in \cite{LPP} furnishes examples of static quasi-Einstein metrics with $\lambda >0$ and $m>1$ for which $X$ is exact but does not vanish. See also the nonhomogeneous Einstein metrics on closed warped product manifolds constructed in \cite{Boehm}.

\item Compare Theorem \ref{theorem1.2}.(i) to a result of Perelman \cite[Sections 2.3 and 2.4]{Perelman}, which shows that all expanding Ricci solitons ($\lambda>0$ and $m=\infty$ in \eqref{eq1.3}) are gradient.

\item The positivity of the Yamabe type for the $\lambda>0$ case has been proved by Lucietti even when $X$ is not closed, but only for $m=2$ \cite{Lucietti}. In the appendix, we will show that his proof actually works whenever $0<m\le\frac{4(n-1)}{(n-2)}$. We will need this fact in the proof of Theorem \ref{theorem1.2}.(i) when $m$ lies between $0$ and $1$.
\end{enumerate}

These results indicate that the condition $dX=0$ is a promising natural condition for rigidity (or at least further classification, in view of the metrics of \cite{LPP} and the metric \eqref{eq1.4}) of quasi-Einstein metrics on closed manifolds. Without this condition, there are many examples of near horizon geometries that are not Einstein, e.g., the near horizon geometries of the extreme Kerr and Myers-Perry black holes. 

In Section \ref{section2}, we review the reduction of static vacuum near horizon geometries to the study of a quasi-Einstein equation for $(M,g,X)$ with $m=2$ together with the condition that $dX=0$. In Section \ref{section3}, we then use this reduced system (but letting $m>0$ be arbitrary) to prove Theorem \ref{theorem1.2}. In Section \ref{section4} we turn our attention to static near horizon geometries in the presence of matter. We will describe the near horizon geometries for these equations and introduce restrictive but physically reasonable assumptions under which those equations reduce to the form of Equation \eqref{eq1.3}. This will allow us to apply Theorem \ref{theorem1.2} to static near horizon geometries with some forms of matter. The appendix contains a version of Lucietti's $m=2$ proof of the positivity of the Yamabe type for near horizon geometries (with general $X$) which holds for $m\le \frac{4(n-1)}{(n-2)}$. This is used to extend the proof of Theorem \ref{theorem1.2}.(i) from $m\ge 1$ to $m>0$.

\subsection{Conventions} Throughout, $(M,g)$ is a connected $n$-manifold. We also take $M$ to be closed (compact without boundary). We will take $X$ to be a closed 1-form on $(M,g)$, except that we relax this condition in the Appendix. We take $\mathbf{g}$ to be a Lorentzian metric obeying the vacuum Einstein equations with cosmological constant whereas $g$ is a Riemannian quasi-Einstein metric; note that we relax the vacuum condition in Section \ref{section4}. Acting on functions, $\Delta_g$ denotes the trace of the Hessian $\hess_g$; on 1-forms, etc, it is the rough (or connection) Laplacian. Similarly, $\divergence_g \omega$ is the trace of the covariant derivative of a 1-form $\omega$. Quantities with raised indices are $g$-metric duals of the same quantities with lower indices, and conversely. We use the symbol $\otimes$ to denote symmetric tensor product. We always assume that $m \in (0,\infty)$ and $\lambda \in \mathbb{R}$ in equation \eqref{eq1.3}.

\subsection{Acknowledgements} The research of EB was partially supported by a Simons Foundation Grant (\#426628, E. Bahuaud). The research of HK was supported by NSERC grant RGPIN--04887--2018. The research of EW was supported by NSERC grant RGPIN--2022--03440. We are grateful to the referees for many suggestions which improved the presentation.

Data statement: Data sharing not applicable to this article as no datasets were generated or analysed during the current study.

Conflict of interest statement: The authors have no conflicts of interest.

\section{The static near horizon geometry equations}
\setcounter{equation}{0}
\label{section2}

\noindent The vacuum Einstein equations are
\begin{equation}
\label{eq2.1}
\ric(\mathbf{g})-\frac12 R \mathbf{g} + \Lambda \mathbf{g}=0.
\end{equation}
These impose conditions on the quantities $Y$, $X_i$, and $g_{ij}$ appearing in \eqref{eq1.1}. From \cite[equation (12) and Section 4]{KL}, these conditions are
\begin{eqnarray}
\label{eq2.2}
\lambda&=& Y-\frac12 |X|_g^2+\frac{1}{2} \divergence_g X,\\
\label{eq2.3}
\lambda g&=&\ric_X^2(g):=\ric(g) +\frac12 \pounds_X g -\frac{1}{2}X\otimes X,\\
\label{eq2.4}
0&=& \Delta_g Y-3 \nabla_X Y -Y\divergence_g X +2Y|X|_g^2-\frac12 |dX|_g^2,\\
\label{eq2.5}
0&=&\nabla_iY-YX_i-X^j\left ( \nabla_iX_j-\nabla_jX_i\right ) +\frac12 \nabla^j \left ( \nabla_iX_j-\nabla_jX_i\right ).
\end{eqnarray}
Here $\nabla$ is the connection of $g$, and $g$ is used to raise and lower indices. If the spacetime cosmological constant is $\Lambda$ then $\lambda=\frac{2\Lambda}{n}$ where $n= \dim M$ (so the spacetime dimension is $n+2$). By $|dX|_g^2$, we mean $\left ( \nabla_i X_j-\nabla_jX_i\right )\left ( \nabla^i X^j-\nabla^jX^i\right )$. To obtain the \emph{static} near horizon geometry equations, we augment this system with the condition
\begin{equation}
\label{eq2.6}
dX=0.
\end{equation}
We may use \eqref{eq2.6} to reduce the above system to
\begin{eqnarray}
\label{eq2.7}
\lambda &=& Y-\frac12 |X|_g^2+\frac12 \divergence_g X,\\
\label{eq2.8}
\lambda g&=&\ric_X^2(g):=\ric(g) +\frac12 \pounds_X g -\frac{1}{2}X\otimes X,\\
\label{eq2.9}
0&=& \Delta_g Y-3 \nabla_X Y -Y\divergence_g X +2Y|X|_g^2,\\
\label{eq2.10}
0&=&\nabla_iY-YX_i, \\
\label{eq2.11}
0&=& dX.
\end{eqnarray}

\begin{lemma}\label{lemma2.1}
Let $(M,g,X)$ be a solution of equations \eqref{eq2.8} and \eqref{eq2.11}. If \eqref{eq2.7} defines $Y=Y(g,X)$ in terms of this solution, then \eqref{eq2.9} and \eqref{eq2.10} are satisfied as well.
\end{lemma}

This result is of course a simplified version of a well-known result for aribtrary $X$; see e.g., \cite[Section 2]{KL}. See also \cite[Proposition 5]{KK}, from which one can obtain a version for closed $X$ with general $m$.

\begin{proof} Applying the operator $\nabla^i -2X^i$ to \eqref{eq2.10} we obtain \eqref{eq2.9}, so \eqref{eq2.9} is redundant and will automatically be satisfied if \eqref{eq2.10} is satisfied.

To see that \eqref{eq2.10} is satisfied when \eqref{eq2.8} is (if \eqref{eq2.7} defines $Y$), let $R$ be the scalar curvature of $g$. Using  equation \eqref{eq2.11} to write $(\frac{1}{2} \pounds_X g)_{ij} = \nabla_i X_j$ and then taking the divergence of \eqref{eq2.8}, applying the contracted second Bianchi identity, and using the Ricci identity, we obtain
\begin{equation}
\label{eq2.12}
\begin{split}
\frac12 \nabla_i R =&\, \frac12 X^j\nabla_j X_i +\frac12 X_i \divergence_g X -R_{ij}X^j -\nabla_i \left ( \divergence_g X\right )\\
=&\, \frac14 \nabla_i \left ( |X|_g^2 \right )  +\frac12 X_i \divergence_g X -R_{ij}X^j -\nabla_i \left ( \divergence_g X\right ),
\end{split}
\end{equation}
using \eqref{eq2.11} to write $\frac12 X^j\nabla_j X_i=\frac12 X^j\nabla_i X_j=\frac14 \nabla_i \left ( |X|_g^2 \right )$ in the last line. Now the trace of \eqref{eq2.8} yields
\begin{equation}
\label{eq2.13}
R=\frac12 |X|_g^2 -\divergence_g X +n\lambda,
\end{equation}
where $n=\dim M$. Using this to substitute for $R$ in \eqref{eq2.12} and also using \eqref{eq2.8} again to replace the $R_{ij}$ term on the right, we obtain (after some cancellations) that
\begin{equation}
\label{eq2.14}
\nabla_i \left ( \divergence_g X -|X|_g^2 -2\lambda \right ) = \left ( \divergence_g X -|X|_g^2 -2\lambda \right )X_i.
\end{equation}
Finally, using \eqref{eq2.7} to define $Y:=-\frac12 \divergence_g X +\frac12 |X|_g^2 +\lambda$, then \eqref{eq2.14} can be written as
\begin{equation}
\label{eq2.15}
\nabla_i Y = YX_i,
\end{equation}
which is \eqref{eq2.10}.
\end{proof}

\section{Static quasi-Einstein manifolds: Proof of Theorem \ref{theorem1.2}}
\setcounter{equation}{0}
\label{section3}

\noindent Lemma \ref{lemma2.1} reduces the study of the static vacuum near horizon geometry equations to the study of the system comprising \eqref{eq2.8} and \eqref{eq2.11}. Equation \eqref{eq2.8} is the $m=2$ case of the quasi-Einstein equation \eqref{eq1.3}. While our interest in Theorem \ref{theorem1.2} is the near horizon geometry case of $m=2$, the theorem is just as easily proved for arbitrary $m>0$, so in this section we replace \eqref{eq2.8} by \eqref{eq1.3} and consider the system comprising \eqref{eq1.3} and \eqref{eq2.11}.

We first need the following lemma which applies when $X$ is an exact form.

\begin{lemma}\label{lemma3.1}
Let the quasi-Einstein equation \eqref{eq1.3} hold on a closed manifold $M$ with $X=df$ and $m>0$. Then
\begin{equation}
\label{eq3.1}
\Delta_g f -|df|_g^2 -m\lambda =-m\mu e^{2f/m}
\end{equation}
for some arbitrary constant $\mu$ (sometimes called the \emph{characteristic constant}).
\end{lemma}

\begin{proof}
If $X=df$ then \eqref{eq2.14} simplifies to
\begin{equation}
\label{eq3.2}
\nabla_i \left [ \Delta_g f -|df|_g^2 -2\lambda \right ] = \left [ \Delta_g f -|df|_g^2 -2\lambda \right ]\nabla_i f,
\end{equation}
so
\begin{equation}
\label{eq3.3}
\nabla_i \left \{ e^{-f} \left [ \Delta_g f -|df|_g^2 -2\lambda \right ] \right \}=0.
\end{equation}
Then \eqref{eq3.1} follows.
\end{proof}

The above lemma will be needed in the proof of the next lemma, from which Theorem \ref{theorem1.2} will follow.

\begin{lemma}\label{lemma3.2}
Let $(X,g)$ solve the quasi-Einstein equation \eqref{eq1.3} on a compact manifold $M$, where $X$ is a closed 1-form $dX=0$ and $m>0$.
\begin{itemize}
\item[(i)] If $\lambda> 0$ then $X$ is an exact form; i.e., $X=df$ for some $f:M\to{\mathbb R}$ and $f$ obeys \eqref{eq3.1}.
\item[(ii)] If $\lambda=0$ then $X$ vanishes and $(M,g)$ is Ricci-flat.
\item[(iii)] If $\lambda< 0$ then either $X$ is exact or $X$ obeys
\begin{equation}
\label{eq3.4}
\divergence_g X -|X|_g^2 = m\lambda .
\end{equation}
\end{itemize}
\end{lemma}

Note that when $X$ is exact and $(M,g)$ is a closed manifold, then \cite{CSW} has enumerated several circumstances for which $(M,g)$ is Einstein, including if $n=2$ (\cite[Theorem 1.2]{CSW}), or if $\lambda=0$ and $m>1$ (\cite[Proposition 3.6.(b)]{CSW}). Part (ii) of the above Lemma improves the latter result to include all $m>0$.

\begin{proof} We generalize an argument made in \cite{CRT} for the $\lambda=0$ case, but which holds for all $\lambda\in{\mathbb R}$. Cover $M$ by contractible domains, so that on each such domain ${\mathcal O}_{\alpha}$ we have $X=df_{\alpha}$. Equation  \eqref{eq3.1} holds for $f_{\alpha}$ in each domain ${\mathcal O}_{\alpha}$, where from equation \eqref{eq3.3} there is a $\mu_{\alpha}\in {\mathbb R}$ such that on ${\mathcal O}_{\alpha}$ we have
\begin{equation}
\label{eq3.5}
\Delta_g f_{\alpha} -|df_{\alpha}|_g^2 -m\lambda =-m\mu_{\alpha} e^{2f_{\alpha}/m}.
\end{equation}
The left-hand side of this expression equals $\divergence_g X - |X|_g^2 - m\lambda$ and is defined independently of $f_{\alpha}$. But on non-empty overlap regions ${\mathcal O}_{\alpha}\cap{\mathcal O}_{\beta}$, $f_{\alpha}$ and $f_{\beta}$ can differ (at most) by a constant $f_{\beta}-f_{\alpha}=C_{\alpha\beta}\in {\mathbb R}$. Thus, from the right-hand side of \eqref{eq3.5}, we have $m\mu_{\alpha} e^{2f_{\alpha}/m}=m\mu_{\beta} e^{2f_{\beta}/m}$ on ${\mathcal O}_{\alpha}\cap{\mathcal O}_{\beta}$, so all the $\mu_{\alpha}$ must have the same sign (either all $>0$, all $<0$, or all $=0$).\footnote
{It follows that $\divergence_g X-|X|_g^2-m\lambda$ is either strictly positive, strictly negative, or identically zero on $(M,g)$. We will see that, for $\lambda<0$, it is zero.}

Assume that the $\mu_{\alpha}$ are not all zero, and so are either all positive or all negative. Then, since $m\neq 0$, we must have $\mu_{\alpha} e^{2f_{\alpha}/m} =\mu_{\beta} e^{2f_{\beta}/m}=\mu_{\beta} e^{2(f_{\alpha}+C_{\alpha\beta})/m}$ on a non-empty overlap region ${\mathcal O}_{\alpha}\cap{\mathcal O}_{\beta}$. But this implies that $\mu_{\alpha}=\mu_{\beta}e^{2C_{\alpha\beta}/m}$. Since the constants $\mu_{\beta}$ are free, we can choose $\mu_{\alpha}=\mu_{\beta}$, and then $C_{\alpha\beta}=0$. Hence $f_{\alpha}=f_{\beta}$ on the overlap. Then we define a smooth function $f$ on ${\mathcal O}_{\alpha}\cup {\mathcal O}_{\beta}$ by $f:=f_{\alpha}$ on ${\mathcal O}_{\alpha}$ and $f:=f_{\beta}$ on ${\mathcal O}_{\beta}$. Iterating, we obtain a globally defined $f$ satisfying \eqref{eq3.1}. Thus, either we obtain a function $f$ such that $X=df$ or the $\mu_{\alpha}$ are all zero. If the $\mu_{\alpha}$ are all zero, then \eqref{eq3.5} yields \eqref{eq3.4}, so \eqref{eq3.4} must hold whenever $X$ is not exact. This holds for all signs of $\lambda$. For $\lambda<0$ it proves Lemma \ref{lemma3.2}.(iii). 

Integrating \eqref{eq3.4} over the closed manifold $M$, we obtain 
\begin{equation}
\label{eq3.6}
-m\lambda = \frac{\int_M|X|_g^2dV}{\vol(M)}.
\end{equation}
Hence \eqref{eq3.4} cannot hold if $\lambda>0$, and if $\lambda=0$ it implies that $X \equiv 0$, so we conclude that $X$ is necessarily exact unless $\lambda<0$. This proves Lemma \ref{lemma3.2}.(i).

If $\lambda=0$, it remains to consider only the case of exact $X$, and then \eqref{eq3.1} applies globally on $M$. But integrating \eqref{eq3.1} on a compact manifold with $\lambda=0$, we obtain $\int_M|df|_g^2 dV = \mu \int_M e^{2f/m} dV$, so $\mu\ge 0$ and if $\mu=0$ then $f=const$, implying that $df\equiv X=0$ and so the quasi-Einstein equation becomes simply $\ric=0$. If, on the other hand, $\mu>0$ then both sides of \eqref{eq3.1} are strictly negative everywhere. This is a contradiction, since $f$ must achieve a minimum on the closed manifold $M$ and the left-hand side of \eqref{eq3.1} cannot be negative at a minimum. This proves Lemma \ref{lemma3.2}.(ii).
\end{proof}

Lemma \ref{lemma3.2}.(i) can also be proved by noting that Myers's theorem holds for $\ric_X^m\ge \lambda>0$ \cite{Limoncu}, and in consequence the fundamental group must be finite, which in turn implies that $b_1(M)=0$. Then \emph{all} closed 1-forms are exact. Then we can again apply Lemma \ref{lemma3.1}.

We can improve on case (iii) of Lemma \ref{lemma3.2}, as we show in the next result.

\begin{lemma}\label{lemma3.3}
Let $(M,g,X)$ solve the quasi-Einstein equation \eqref{eq1.3} with $\lambda<0$, $M$ a closed manifold, and $X$ a closed 1-form $dX=0$. Further assume that equation \eqref{eq3.4} holds. Then $X$ is incompressible (i.e., $\divergence_g X=0$) and has constant norm $|X|_g^2=-m\lambda$. The Euler characteristic of $M$ is zero. The scalar curvature of $(M,g)$ is $R=\left ( n-1\right )\lambda$, and therefore $\dim M :=n \neq 2$.
\end{lemma}

\begin{proof}
Note that by the Ricci identity we have $\nabla_i \divergence_g X = \nabla_j\nabla_i X^j -R_{ij}X^j$. Furthermore, $\nabla_j\nabla_i X^j\equiv \nabla^j\nabla_i X_j= \nabla^j\nabla_j X_i \equiv \Delta_g X_i$ since $\nabla_i X_j = \nabla_j X_i$ (i.e., since $dX=0$). Here $\Delta_g X_i$ is the rough (or connection) Laplacian (i.e. $\tr_g \nabla^2$) acting on $X_i$. By differentiating \eqref{eq3.4} and commuting derivatives as we have just done, we obtain
\begin{equation}
\label{eq3.7}
\begin{split}
0=&\, \nabla \left ( \divergence_g X -|X|_g^2-m\lambda\right )\\
=&\, \Delta_g X - \ric(\cdot, X) -\nabla \left ( |X|_g^2 \right ).
\end{split}
\end{equation}
Contracting against $2X$ and using that $2X^i\Delta_g X_i=\Delta_g \left ( |X|_g^2\right )-2\left \vert \nabla X \right \vert_g^2$, we obtain
\begin{equation}
\label{eq3.8}
0=\Delta_g \left ( |X|_g^2 \right ) -2| \nabla X|_g^2 -2\ric(X,X) -2\nabla_X  \left ( |X|_g^2 \right ).
\end{equation}
But from \eqref{eq1.3} we have that
\begin{equation}
\label{eq3.9}
\ric(X,X)=-\frac12 \nabla_X  \left ( |X|_g^2 \right ) +\frac{1}{m}\left ( |X|_g^2 \right )^2 +\lambda |X|_g^2.
\end{equation}
We insert this into \eqref{eq3.8} to obtain
\begin{eqnarray}
\label{eq3.10}
& & 0=\Delta_g \left ( |X|_g^2 \right ) -\nabla_X  \left ( |X|_g^2 \right ) -2|\nabla X|_g^2  -\frac{2}{m} |X|_g^2  \left ( |X|_g^2 +m\lambda \right )\\
\label{eq3.11}
& \implies & \Delta_g \left ( |X|_g^2 \right ) -\nabla_X  \left ( |X|_g^2 \right ) -2|\nabla X|_g^2 =\frac{2}{m} |X|_g^2  \left ( |X|_g^2 +m\lambda \right ).
\end{eqnarray}
Since $|X|_g^2$ must achieve a maximum on the closed manifold $M$ and since every term on the left in \eqref{eq3.11} is nonpositive at a maximum, the right-hand side must also be nonpositive at a maximum of $|X|_g^2$. Therefore $|X|_g^2 \le -m\lambda$ at the maximum of $|X|_g^2$, and hence everywhere. But from \eqref{eq3.6} we see that the average value of $|X|_g^2$ on $M$ is $-m\lambda$. Since the maximum of $|X|_g^2$ equals its average, it must be constant and so $|X|_g^2 \equiv -m\lambda$ on $M$. Then from \eqref{eq3.4} we have $\divergence_g X=0$. Then from the trace of equation \eqref{eq1.3}, we get that $R=\left ( n-1\right )\lambda$. The Euler characteristic must vanish because $|X|$ is never zero. But since $\int_M R dV = (n-1)\lambda \vol(M)<0$, we have a contradiction when $n=2$.
\end{proof}

\medskip

Now we prove the main result.

\begin{proof}[Proof of Theorem \ref{theorem1.2}] 
To prove Theorem \ref{theorem1.2}.(i), choose $\lambda>0$. Then by Lemma \ref{lemma3.2}.(i) we know that $X=df$. Then we may take the trace of equation \eqref{eq1.3} and use $X=df$ to obtain  
\begin{equation}
\label{eq3.12}
R+\Delta_g f -\frac{1}{m} \vert df \vert_g^2 = n \lambda.
\end{equation}
Multiplying this equation by $e^{-f/m}$ and integrating over the closed manifold $M$, we obtain $\int_M (R-n\lambda)e^{-f/m}dV=0$. If $R$ is constant then we must have $R=n\lambda$, so \eqref{eq3.12} becomes $\Delta_g f - \frac{1}{m}|df|_g^2=0$. Then the strong maximum principle implies that $f=const$, so $X=0$ and then the quasi-Einstein equation \eqref{eq1.3} becomes simply $\ric=\lambda g$. 

To prove the assertion concerning the Yamabe type, first let $m\ge 1$. Then we may use that $X$ is exact to invoke \cite[Proposition 3.6.(a)]{CSW}, which states that $R\ge \frac{n(n-1)}{(n+m-1)}\lambda>0$. Then a standard argument, using only that $(M,g)$ is closed and $R>0$, shows that the Yamabe constant of the conformal class of the metric $g$ must be positive and hence so is the Yamabe invariant (the supremum over conformal classes of the Yamabe constant of each class). For $0<m<1$ and $n>2$, we may invoke Proposition \ref{propositionA.1}. Finally, for $n=2$, the Yamabe type is the sign of the Euler characteristic $\chi(M)$. We take the trace of \eqref{eq1.3} and integrate to get $2\pi\chi(M)=\int_M RdV_g = \lambda \vol(M)+\frac{1}{m}\int_M |X|_g dV_g>0$. This proves Theorem \ref{theorem1.2}.(i).

Theorem \ref{theorem1.2}.(ii) follows immediately from Lemma \ref{lemma3.2}.(ii) so there is nothing to do. 

To prove Theorem \ref{theorem1.2}.(iii), assume that $X$ is exact. Then equation \eqref{eq3.1} holds. Then at a maximum, say $f=f_{\rm max}$, we have $\Delta_g f \le 0$. From \eqref{eq3.1} then $-m\lambda\ge -m\mu e^{2f_{\rm max}/m}$. At a minimum, say $f=f_{\rm min}$, we have $\Delta_g f \ge 0$ so $-m\lambda \le -m \mu e^{2f_{\rm min}/m}$. Putting these together we have $-m\mu e^{2f_{\rm max}/m}\le -m\mu e^{2f_{\rm min}/m}$. Now $\mu$ has the same sign as $\lambda$. [Proof: multiply \eqref{eq3.1} by $e^{-f}$ and integrate to get $-m\lambda\int e^{-f}dV = -m\mu\int e^{(2-m)f/m}dV$.] So $m\mu<0$. Then from $-m\mu e^{2f_{\rm max}/m}\le -m\mu e^{2f_{\rm min}/m}$ we get that $e^{2f_{\rm max}/m}\le e^{2f_{\rm min}/m}$, so $f_{\rm max}/m\le f_{\rm min}/m$. If $m>0$ this yields $f_{\rm max}\le f_{\rm min}$, so $f$ is constant and $g$ is Einstein. If $b_1(M)=0$ then $X$ is necessarily exact, so then $g$ is Einstein. But if $X$ is not exact, then by Lemma \ref{lemma3.2}.(iii) we must have that $\divergence_g X -|X|_g^2=m\lambda$ and then by Lemma \ref{lemma3.3} we have $\divergence_g X =0$, $|X|_g^2=-m\lambda$, $\chi(M)=0$, $R=\left ( n-1\right )\lambda$, and $n\neq 2$.
\end{proof}

We note that the argument given for the case of $X$ exact in the last paragraph of the proof is essentially that used in the proof of \cite[Theorem 1.1]{KK}, which used the quasi-Einstein equation with $m$ an integer to study warped products.

We end this section by mentioning some other results which strengthen Theorem \ref{theorem1.2} in low dimensional cases.

When $\lambda >0$ and $M$ is closed, Theorem \ref{theorem1.2}.(i) asserts that the closed 1-form $X$ is exact. In \cite{CRT} it is shown that if one further assumes that $n=m=2$, the 1-form vanishes and $(M,g)$ is Einstein. It was later shown in \cite[Theorem 1.2]{CSW} that if $M$ is exact then the $n=m=2$ condition could be relaxed to $m>0$ with $n=2$. Hence, we can strengthen Theorem \ref{theorem1.2}.(i): For $X$ a closed 1-form and $M$ a closed 2-manifold, any quasi-Einstein metric is Einstein.

If we dispense for the moment with the assumption that the 2-manifold $M$ is closed but fix $\lambda\neq 0$, we can integrate \eqref{eq3.1} with exact $X$. This calculation was carried out in \cite[Appendix A]{HPW} when $m> 1$. Their result can be expressed in a familiar coordinate system as
\begin{equation}
\label{eq3.13}
ds^2 = \frac{d\psi^2}{\left ( 1  -\frac{a^{m-1}}{\psi^{m-1}} -\frac{(m-1)\lambda}{(m+1)\mu}\psi^2\right )} +\left ( 1  -\frac{a^{m-1}}{\psi^{m-1}} -\frac{(m-1)\lambda}{(m+1)\mu}\psi^2\right ) d\tau^2,
\end{equation}
where $a\in {\mathbb R}$ is an arbitrary constant, and for completeness $X=d(-m\log \psi)$. For $m$ an integer, readers will recognize the metric as the quotient of an $(m+2)$-dimensional Schwarzschild-de Sitter metric (with Riemannian signature) by its $\SO(m+1)$ isometry group. It is also clearly a metric on a cylinder that cannot be smoothly compactified to ${\mathbb S}^2$ unless $a=0$, and then it is a round metric. This gives another perspective on \cite[Theorem 1.2]{CSW}. The calculation is also valid for $0<m<1$, but then the $a=0$ case is a hyperbolic metric.

Finally, for $n=3$ and $\lambda>0$, the condition on the Yamabe type permits nontrivial quasi-Einstein metrics with exact $X$ on spherical 3-spaces, ${\mathbb S}^2\times {\mathbb S^1}$, and connected sums thereof (we can restrict to orientable cases since a quasi-Einstein structure on a non-orientable manifold will induce one on the orientable double cover). But by a Bakry-\'Emery version of the Myers theorem, the fundamental group must be finite \cite{Limoncu}. Hence the universal cover of a closed quasi-Einstein 3-manifold with $\lambda>0$ and $dX=0$ is diffeomorphic to ${\mathbb S}^3$ (but need not be round).

\section{Static quasi-Einstein manifolds with matter}
\setcounter{equation}{0}
\label{section4}

\noindent The Einstein equation with matter in $(n+2)$-dimensional spacetime $(\mathbf{M}, \mathbf{g})$ is written as \begin{equation}
\label{eq4.1}
\ric(\mathbf{g})-\frac12 R \mathbf{g}+\Lambda \mathbf{g} = {\mathbf{T}}\ ,
\end{equation}
where $R:=\tr \ric(\mathbf{g})$ is the scalar curvature of spacetime, $\Lambda=\frac{n\lambda}{2}$, and the stress-energy tensor $\mathbf{T}$ is a divergenceless symmetric rank-2 tensor field which characterizes the matter (up a a numerical factor which we have absorbed). Equation \eqref{eq1.4} may be considered to be an inhomogeneous version of the vacuum equations \eqref{eq2.1}.  In the coordinates of \eqref{eq1.1}, the near-horizon limit of $\mathbf{T}$ may be written as \cite[equation 16]{KL}
\begin{equation}
\label{eq4.2}
{\mathbf T} \to 2dv\left [ T_{+-}dr +r\left ( \beta_i +T_{+-}X_i\right ) dx^i +\frac{r^2}{2}\left ( T_{+-}Y -\frac12 \divergence_g \beta + X_i\beta^i \right )dv\right ] +T_{ij}dx^i dx^j,
\end{equation}
where $T_{+-}$, $\beta$, and $T$ are a function, 1-form, and symmetric $(0,2)$-tensor, respectively, on $M$, with $\beta$ expressible in terms of $T_{+-}$ and $T$ through the formula
\begin{equation}
\label{eq4.3}
\beta_i := -\nabla^jT_{ij}+T_{ij}X^j -T_{+-}X_i.
\end{equation} 
The static condition for the ambient spacetime implies $K \wedge \mathbf{T}(K) =0$ where $K = \mathbf{g}(\partial_v,\cdot)$ in the coordinates of \eqref{eq1.1} \cite{KL2}. It follows using \eqref{eq4.2} that $\beta_i=0$. Thus, when matter is included, the static near horizon equations \eqref{eq2.7}--\eqref{eq2.11} are then replaced by
\begin{eqnarray}
\label{eq4.4}
\lambda&=& Y-\frac12 |X|_g^2 + \frac12 \divergence_g X -\frac{(n-2)}{n}T_{+-}+\frac{\tr_gT}{n},\\
\label{eq4.5}
\lambda g&=&\ric_X^2(g)-P:=\ric(g) +\frac12 \pounds_X g -\frac{1}{2}X\otimes X -P,\\
\label{eq4.6}
0&=& \Delta_g Y-3 \nabla_X Y -Y\divergence_g X +2Y|X|_g^2  ,\\
\label{eq4.7}
0&=&\nabla_iY-YX_i, \\
\label{eq4.8}
0&=& dX,\\
\label{eq4.9}
0&=& -\nabla^jT_{ij}+T_{ij}X^j -T_{+-}X_i,\\
\label{eq4.10}
P&:=& T-\frac{1}{n} \left ( \tr_g T +2T_{+-}\right )g.
\end{eqnarray}
Note that \eqref{eq4.10} is merely a definition so we can immediately substitute this into the preceding equations to reduce the system. And similarly to the vacuum case, \eqref{eq4.6} can be obtained by operating with $\nabla^i-2X^i$ on \eqref{eq4.7}. 

We now derive a version of Lemma \ref{lemma2.1} valid for the static near horizon geometry equations with matter.

\begin{lemma}\label{lemma4.1}
Let $(M,g,X)$ be a solution of equations \eqref{eq4.5} and \eqref{eq4.8} for a given stress-energy tensor $T$. If \eqref{eq4.4} defines $Y=Y(g,X,T)$ in terms of this solution, then \eqref{eq4.6} and \eqref{eq4.7} are satisfied as well.
\end{lemma}

\begin{proof} Repeat the steps of Section 2, beginning this time from \eqref{eq4.5}. We apply the contracted Bianchi identity and Ricci identity as before. This yields a version of \eqref{eq2.14} valid when matter is present:
\begin{equation}
\label{eq4.11}
\nabla_i \left ( \divergence_g X -|X|_g^2 -2\lambda \right ) -\left ( \divergence_g X -|X|_g^2 -2\lambda \right )X_i = 2\nabla^j P_{ij}-2P_{ij}X^j -\nabla_i \tr_g P.
\end{equation}
Treating \eqref{eq4.4} as the definition of $Y$, and using \eqref{eq4.10}, this reduces to 
\begin{equation}
\label{eq4.12}
-2\left ( \nabla_i Y -YX_i\right ) =0,
\end{equation}
which is \eqref{eq4.7} and then, as stated above, \eqref{eq4.6} follows from application of $\nabla^i-2X^i$ to \eqref{eq4.7}. 
\end{proof} 

We therefore can restrict our attention to equations \eqref{eq4.5} and \eqref{eq4.8} (incorporating as well the definition \eqref{eq4.10}):
\begin{eqnarray}
\label{eq4.13}
\lambda g&=& \ric(g) +\nabla X -\frac{1}{2}X\otimes X -\left [ T-\frac{1}{n} \left (\tr_g T\right )g\right ] +\frac{2}{n}T_{+-}g,\\
\label{eq4.14}
0&=& (dX)_{ij} = \frac12 \left ( \nabla_i X_j -\nabla_j X_i\right ) .
\end{eqnarray}
In the special case of $T_{+-}=const$ and $\tf_g T := T-\frac{1}{n} \left (\tr_g T\right )g=0$, equation \eqref{eq4.13} reduces to
\begin{equation}
\label{eq4.15}
{\tilde \lambda} g = \ric(g) +\nabla X -\frac{1}{2}X\otimes X,
\end{equation}
which has the same form as \eqref{eq2.8} with $m=2$. Then Theorem \ref{theorem1.2} immediately yields the following result.

\begin{theorem}\label{theorem4.2}
Let $(M,g,X)$ be a compact static near horizon geometry (equations \eqref{eq4.4}--\eqref{eq4.10}) with stress-energy tensor given by equation \eqref{eq4.2} such that the tracefree part of $T_{ij}$ vanishes and $T_{+-}$ is constant. Let ${\tilde \lambda} =\lambda-\frac{2}{n}T_{+-}$. Then the conclusions of Theorem \ref{theorem1.2} hold, with $m=2$ and ${\tilde \lambda}$ replacing $\lambda$.
\end{theorem} 
It is straightforward to construct solutions of \eqref{eq4.4} by considering the stress energy tensor produced by a Maxwell field (an Abelian gauge theory). Such matter is described by a closed 2-form $\mathbf{F}$ in spacetime $(\mathbf{M}, \mathbf{g})$ with associated stress tensor
\begin{equation}
    \mathbf{T}_{\mu\nu} = 2 \left(\mathbf{F}_{\mu \rho} \mathbf{F}_{\nu}^{~\rho} - \frac{\mathbf{g}_{\mu\nu}}{4} |\mathbf{F}|^2\right)
\end{equation} where $|\mathbf{F}|^2 = \mathbf{F}_{\mu\nu} \mathbf{F}^{\mu\nu}$.  Consider $M = \mathbb{S}^n$ equipped with its canonical round metric with some radius $\ell$, $X=0$, and  $\mathbf{F} = d (c r dv)$ in the coordinates of \eqref{eq1.1} for some constant $c$, It is easy to check that the only non-vanishing part of $\mathbf{T}$ is $T_{+-} = -c^2$. One then easily obtains a solution with an appropriate choice of $\ell$.

An example of a static quasi-Einstein metric with matter with $X \neq 0$ when $n=3$,  which generalizes \eqref{eq1.4}, can be obtained by endowing $M = \mathbb{S}^1 \times \Sigma$ with the metric
\begin{equation}
    g = (1 + k^2) d\Phi^2 + g_\Sigma
\end{equation} where $\Phi$ is identified with period $2\pi$ and $g_\Sigma$ is a metric of constant curvature on some two-dimensional compact manifold $\Sigma$. For concreteness, fix $\lambda =-2$. If we choose $\mathbf{F} = \sqrt{3} k d\text{Vol}_\Sigma$ where $d\text{Vol}_\Sigma$ denotes the volume-form associated to $g_\Sigma$, it is easy to check that $T_{+-} =0$ and $T = 3 k^2 g_\Sigma$. Taking $X = 2(1+k^2) d\Phi$ we arrive at a solution of the quasi-Einstein equations with sources provided $g_\Sigma$ has sectional curvature $-2(1 - 2k^2)$. In particular by choosing $k$ appropriately, we may choose $(\Sigma,g)$ to be either a round 2-sphere,  a flat torus, or a compact hyperbolic space, in contrast to our previous example \eqref{eq1.4}. 

\appendix

\section{Lucietti's argument for the Yamabe type}
\setcounter{equation}{0}

\noindent In \cite{Lucietti}, Lucietti observed that it's easy to find circumstances under which the Yamabe invariant of a near horizon geometry is positive or nonnegative (the sign of this invariant is the \emph{Yamabe type} of the closed mainfold $M$). Here we extend his observation to general $m$. Although we use this in the proof of Theorem \ref{theorem1.2}.(i) for closed 2-form $X$, Lucietti's result and our generalization holds for arbitrary $X$.

Let $n\ge 3$. Taking the trace of \eqref{eq1.3} and multiplying by $\phi^2$, we have
\begin{equation}
\label{eqA.1}
\begin{split}
\frac{4(n-1)}{(n-2)}|\nabla \phi |^2 +R_g\phi^2 =&\, \frac{4(n-1)}{(n-2)}|\nabla \phi |^2 +\left ( -\divergence X +\frac{1}{m}|X|^2 +n\lambda \right )\phi^2\\
=&\, \frac{4(n-1)}{(n-2)}|\nabla \phi |^2-\divergence \left ( \phi^2 X\right ) +2\phi X\cdot\nabla\phi +\frac{1}{m}|X|^2\phi^2 +n\lambda\phi^2\\
=&\, \left ( \frac{4(n-1)}{(n-2)}-k^2\right ) |\nabla \phi |^2-\divergence \left ( \phi^2 X\right ) +\left \vert k\nabla\phi +\frac{1}{k}X\phi\right \vert^2\\ 
&\,  +\frac{\left ( k^2-m\right )}{mk^2} |X|^2\phi^2 +n\lambda\phi^2,
\end{split}
\end{equation}
where $k\neq 0$ is any nonzero constant. Inserting this in \eqref{eq1.7}, we obtain
\begin{equation}
\label{eqA.2}
Q_g(\phi) = \frac{\int_M \left \{ \left ( \frac{4(n-1)}{(n-2)}-k^2\right )|\nabla \phi |^2 + \left \vert k\nabla\phi +\frac{1}{k}X\phi\right \vert^2 +\left (\frac{\left ( k^2-m\right )}{mk^2} |X|^2 +n\lambda\right )\phi^2 \right \}dV_g }{\left ( \int_M \phi^{\frac{2n}{(n-2)}}dV_g\right )^\frac{n-2}{n}}.
\end{equation}

\begin{proposition}\label{propositionA.1} Let $(M,g)$ be a closed orientable $n$-manifold with $n\ge 3$, obeying \eqref{eq1.3} with $\lambda\ge 0$. If $0<m < \frac{4(n-1)}{(n-2)}$, then the Yamabe invariant of $M$ is positive. If as well $n=3$, then $M$ is diffeomorphic either to a spherical space or to $S^2\times S^1$, or to connected sums of the these. If instead $m = \frac{4(n-1)}{(n-2)}$ then Yamabe type zero is possible.
\end{proposition}
\begin{proof} 
If $\lambda\ge 0$ and $0<m\le k^2<\frac{4(n-1)}{(n-2)}$, then \eqref{eqA.2} yields
\begin{equation}
\label{eqA.3}
Q_g(\phi)\ge  \min\left\{ \frac{4(n-1)}{(n-2)}-k^2, n \lambda \right\} \frac{\int_M |\nabla \phi |^2  + \phi^2 \, dV_g }{\left ( \int_M \phi^{\frac{2n}{(n-2)}}\, dV_g\right )^\frac{n-2}{n}}.
\end{equation}
abe constant of $(M,g)$ obeys $Y(M,[g]):=\inf_{\phi\in C^{\infty}(M)\backslash \{ 0 \}} Q_g(\phi)>0$, In turn, the Yamabe invariant of $M$ is the supremum of $Y(M,[g])$ over all conformal classes $[g]$, so it must be positive as well. This implies the stated topological restrictions when $n=3$ and proves the theorem when $0<m<\frac{4(n-1)}{(n-2)}$. 

If $\lambda\ge 0$ and $m=\frac{4(n-1)}{(n-2)}$, choose $k^2=m=\frac{4(n-1)}{(n-2)}$ in \eqref{eqA.2}. Then \eqref{eqA.2} is manifestly nonnegative, proving the $m=\frac{4(n-1)}{(n-2)}$ case .
\end{proof}


\begin{thebibliography}{99}
\bibitem{Besse} AL Besse, \emph{Einstein manifolds}, Ergebnisse der Mathematik und ihrer Grenzgebiete 3 Vol 10 (Springer, Berlin, 1987).
\bibitem{Boehm} C B\"ohm, \emph{Inhomogeneous Einstein metrics on low-dimensional spheres and other low-dimensional spaces}, Invent Math 134 (1998) 145--176.
\bibitem{Case} JS Case, \emph{On the nonexistence of quasi-Einstein metrics}, Pac J Math 248 (2010) 277--284.
\bibitem{CSW} JS Case, Y-J Shu, and G Wei, \emph{Rigidity of quasi-Einstein metrics}, Diff Geom Appl 29 (2011) 93--100.
\bibitem{CRT} PT Chru\'sciel, HS Reall, and P Tod, \emph{On non-existence of static vacuum black holes with degenerate components of the event horizon}, Class Quantum Gravit 23 (2006) 549--554.
\bibitem{FFT} AA Freitas Filho and K Tenenblat, \emph{On generalized quasi-Einstein manifolds}, preprint [arxiv:2112:04301].
\bibitem{GL} M Gromov and B Lawson, \emph{Positive scalar curvature and the Dirac operator on complete Riemannian manifolds}, Publ Math IHES 58 (1983) 83--196.
\bibitem{HPW} C He, P Petersen, and W Wylie, \emph{On the classification of warped product Einstein metrics}, Commun Anal Geom 20 (2012) 271--311.
\bibitem{KW} M Khuri and E Woolgar, \emph{Nonexistence of extremal de Sitter black rings}, Class Quantum Gravit 34 (2017) 22LT01.
\bibitem{KK} D-S Kim and Y H Kim, \emph{Compact Einstein warped product spaces with nonpositive scalar curvature},
Proc Amer Math Soc 131 (2003) 2573--2576.
\bibitem{KL} HK Kunduri and J Lucietti, \emph{Classification of near-horizon geometries of extremal black holes}, Living Reviews in Relativity 16--8 (2013).
\bibitem{KL2} HK Kunduri and J.~Lucietti,\emph{Uniqueness of near-horizon geometries of rotating extremal AdS(4) black holes} Class. Quant. Grav. \textbf{26} (2009), 055019.
\bibitem{Lim} A Lim, \emph{Locally homogeneous non-gradient quasi-Einstein 3-manifolds}, Advances in Geometry 22 (2022) 79--93. 
\bibitem{Limoncu} M Limoncu, \emph{Modifications of the Ricci tensor and applications}, Arch Math 95 (2010) 191--199.
\bibitem{Lucietti} J Lucietti, \emph{Two remarks on near-horizon geometries}, Class Quantum Gravit 29 (2012) 235014.
\bibitem{LPP} H Lu, DN Page and CN Pope, \emph{New inhomogeneous Einstein metrics on sphere bundles over Einstein-Kaehler manifolds}, Phys Lett B 593 (2004), 218--226.
\bibitem{Perelman} G Perelman, \emph{The entropy formula for the Ricci flow and its geometric applications}, preprint [arxiv:math.DG/0211159].
\bibitem{SY} R Schoen and S-T Yau, \emph{On the structure of manifolds with positive scalar curvature}, Manuscripta Math 28 (1979) 159--183.
\end{thebibliography}
\end{document}